\documentclass[11pt]{amsart}
\usepackage{amsmath, amssymb, array}
\usepackage[top=30truemm,bottom=30truemm,left=30truemm,right=30truemm]{geometry} 
\usepackage{mathtools}
\usepackage{thmtools}
\usepackage[all]{xy}
\usepackage{ascmac}
\usepackage{mathtools}
\usepackage{mathrsfs}
\usepackage[dvipdfmx]{graphicx}
\usepackage{color}
\usepackage[dvipdfmx]{xcolor}
\usepackage{amscd}
\usepackage{fancyhdr}
\usepackage[bookmarks=false]{hyperref} %for arXiv
\usepackage{amsrefs}
\usepackage{cleveref}
\usepackage{chngcntr}
\usepackage{apptools}
\usepackage{comment}
\usepackage{listings}
\usepackage{booktabs}
\usepackage{multirow}
\usepackage{url}
\usepackage[shortlabels]{enumitem}

\DeclareMathOperator{\rank}{rank} %rank
\DeclareMathOperator{\GL}{GL}

\DeclareMathOperator{\PTE}{PTE}

\renewcommand{\SS}{SS}
\DeclareMathOperator{\lin}{lin}
\DeclareMathOperator{\diag}{diag}

\title{On a group of normalized solutions of the higher-dimensional Prouhet--Tarry--Escott problem}

\author{Munenori Inagaki$^{*}$}
\email{2225063t@stu.kobe-u.ac.jp}
\address{${}^{**}$Graduate School of System Informatics, Kobe University, 1-1 Rokkodai, Nada, Kobe 657-8501, Japan}

\author{Hideki Matsumura$^{**}$}
\email{hmatsumura@tmu.ac.jp}
\address{${}^*$Graduate School of Science, Tokyo Metropolitan University, 1-1 Minami-Osawa, Hachioji-shi, Tokyo 192-0397, Japan}

\author{Masanori Sawa$^{*}$}
\email{sawa@people.kobe-u.ac.jp}
 \address{${}^{**}$Graduate School of System Informatics, Kobe University, 1-1 Rokkodai, Nada, Kobe 657-8501, Japan}

\author{Yukihiro Uchida$^{**}$}
\email{yuchida@tmu.ac.jp}
\address{${}^{*}$Graduate School of Science, Tokyo Metropolitan University, 1-1 Minami-Osawa, Hachioji-shi, Tokyo 192-0397, Japan}

 \thanks{This research is supported by KAKENHI 20K03517, KAKENHI 24KJ0183, KAKENHI 25K17234 and KAKENHI 50508182 of the Japan Society for the Promotion of Science (JSPS), and by the Early Support Program for Grant-in-Aid for Scientific Research of Kobe University.} 



\date{\today}

\theoremstyle{plain}
 \newtheorem{theorem}{Theorem}[section] 
  \crefname{theorem}{Theorem}{Theorems}
 \newtheorem{proposition}[theorem]{Proposition}
 \crefname{proposition}{Proposition}{Propositions}
 \newtheorem{lemma}[theorem]{Lemma}
 \crefname{lemma}{Lemma}{Lemmas}
 
  \crefname{corollary}{Corollary}{Corollaries}
 
   \crefname{conjecture}{Conjecture}{Conjectures}
 
 \crefname{question}{Question}{Questions}
 \newtheorem{problem}[theorem]{Problem}
   \crefname{problem}{Problem}{Problems}
 
    \crefname{notation}{Notation}{Notations}
\crefname{table}{Table}{Tables}

\theoremstyle{definition} 
 \newtheorem{definition}[theorem]{Definition}
  \crefname{definition}{Definition}{Definitions}
 \newtheorem{example}[theorem]{Example}
   \crefname{example}{Example}{Examples}
 \newtheorem{remark}[theorem]{Remark}
   \crefname{remark}{Remark}{Remarks}
    
   \crefname{claim}{Claim}{Claims}

\begin{document}

%% Title&Author

\maketitle

\tableofcontents
%% Abstract 

\begin{abstract}
We elucidate, for the first time, a novel group-theoretic structure that arises from certain solutions of the $n$-dimensional Prouhet--Tarry--Escott problem of degree $2$ and size $n$.
We prove that the group is isomorphic to the orthogonal group for a certain quadratic form.
\end{abstract}

%%%%%%%%%%%%%%%%%%%%%%%%%%%%%%%%%%%%%%%%%%%%%%%%%
%%%%%%%%%%%%%%%%%%%%%%%%%%%%%%%%%%%%%%%%%%%%%%%%%
\section{Introduction} \label{sect:Intro}

In this paper, we consider the Diophantine problem called {\it the $r$-dimensional Prouhet--Tarry--Escott problem} that was introduced by Alpers--Tijdeman \cite{Alpers-Tijdeman}.

\begin{problem} [{\cite{Alpers-Tijdeman}}]  \label{PTEr}
The {\it $r$-dimensional PTE problem
 $(\PTE_r)$ of degree $n$ and size $m$}, asks whether there exists a disjoint pair of multisets
\begin{align} \label{eq:rPTE1}
A=\{(a_{11}, \ldots, a_{1r}), \ldots, (a_{n1}, \ldots, a_{nr})\}, 
B=\{(b_{11}, \ldots, b_{1r}), \ldots, (b_{n1}, \ldots, b_{nr})\} \subset \mathbb{Q}^r 
\end{align}
such that
\[
\sum_{i=1}^n a_{i1}^{k_1} \cdots a_{ir}^{k_r}=\sum_{i=1}^n b_{i1}^{k_1} \cdots b_{ir}^{k_r} 
\quad (1 \leq k_1+ \cdots +k_r \leq m).
\]
The notation $[A]=^n_m [B]$ or $[\mathbf{a}_1,\ldots,\mathbf{a}_n] =_m [\mathbf{b}_1,\ldots,\mathbf{b}_n]$ is used to denote the solution $(A,B)$.
\end{problem}

Alpers--Tijdeman \cite{Alpers-Tijdeman} established a geometric framework for $\PTE_2$ including the explicit construction of ideal solutions.
This framework was generalized to $\PTE_{\geq 3}$ by Ghiglione~\cite{Ghiglione} via geometric techniques related to discrete tomography and techniques involving Gr\"{o}bner bases.
Matsumura and Sawa \cite{MS2025+} constructed a family of rational ideal solutions from a certain combinatorial object called \emph{ellipsoidal design}; for the definition of ellipsoidal designs, see Definition~\ref{ED}.

\begin{definition} [{Normalized solution}] \label{def:normalizedsol}
Suppose that $[A]=_m^n [B]$ is a (not necessarily disjoint) solution of $\PTE_r$ and $\rank A=\rank B=r$, where 
$A$ and $B$ are regarded as the $r \times n$ matrices  whose columns are 
$(a_{i1},\ldots,a_{ir})^T$ and $(b_{i1},\ldots,b_{ir})^T$ respectively.
Then, $(A,B)$ is called a {\it normalized solution} if the $r \times 2n$ matrix $[A \; B] 
$ concatenating $A$ and $B$ is in the reduced row echelon form. 
\end{definition}

\begin{example} \label{ex:NS1}
Let $G$ be a cyclic group generated by $(x_1,x_2, \ldots, x_7) \mapsto (x_7,x_1,\ldots,x_6)$ and $\mathbf{x} ^G$ be the $G$-orbit of a vector $\mathbf{x} \in \mathbb{Q}^7$. Define
\begin{align*} 
X = (1,1,0,1,0,0,0)^G, \; Y=(0,0,1,0,1,1,0)^G.
\end{align*}
Then $[X]=^7_2 [Y]$ is a solution of $\PTE_7$, where $X$ is the set of characteristic vectors of the Fano plane and $Y$ is the `reverse' of the sequences (vectors) of $X$.
The reduced echelon form of $[X \; Y]$ is given by
\begin{align*} 
A=(1,0,0,0,0,0,0)^G, \; B= \left(-\frac{1}{2},\frac{1}{2},\frac{1}{2},0,\frac{1}{2},0,0 \right)^G.
\end{align*}
Therefore, $[A]=^7_2 [B]$ is a normalized solution.
\end{example}

The full rank condition of $A$ and $B$ in \cref{N2n} is natural because most of the solutions including those of $\PTE_2$
in \cite{Alpers-Tijdeman,MS2025+} satisfy this condition.
In particular, Alpers--Tijdeman \cite[Theorem 8]{Alpers-Tijdeman} established a construction 
method of solutions of $\PTE_2$ from points which have equal $X$-rays along with different directions. 
For $(p,q) \in \mathbb{Z}^2$ and $A, B \subset \mathbb{Z}^2$, $A$ and $B$ \emph{have equal $X$-rays along with $\lin \{ (p,q)\}$}, where $\lin \{ (p,q)\} = \{ (\lambda p, \lambda q) \mid \lambda \in \mathbb{R} \}$, if they have an equal number of points on each line parallel to $\lin \{ (p,q)\}$.
This construction was generalized to $\PTE_r$ by Ghiglione~\cite[Theorems 4.3.17 and 4.3.20]{Ghiglione}.

In this paper, we elucidate, for the first time, a novel group-theoretic structure that arises from normalized solutions of $\PTE_n$ of degree $2$ and size $n$ (\cref{N2(n)}).

\begin{definition} [{Group of normalized solutions}]  \label{N2n}
Define the group $N_2(n)$ of {\it normalized solutions} of $\PTE_n$ of degree $2$ and size $n$ by
\[
N_2(n) = \{ A^{-1}B \in M_n(\mathbb{Q}) \mid A, B \in M_n(\mathbb{Q}), \; [A]=_2^n [B], \; \rank A=\rank B=n\}.
\]
\end{definition}

Since $A$ is a nonsingular matrix, $N_2(n)$ is well-defined.
From now on, we identify a normalized solution $[A]=^n_2 [B]$ of $\PTE_n$ with the matrix $A^{-1}B \in N_2(n)$.
By \cref{N2(n)}, $N_2(n)$ forms a group under matrix multiplication.

In \cite{TWZ2025}, the right hand side of \cref{N2(n)} was independently investigated
in the context of spectral graph theory.
\begin{theorem}[{\cite[Theorem 4 and Corollary 1]{TWZ2025}}] \label{TWZ}
Let $n \geq 2$.
Let $\SS_n(\mathbb{Q})$ denote the set of all skew-symmetric matrices of order $n$ and $\Sigma_n(\mathbb{Q})$ denote the set of permutation matrices of order $n$. 
Then for all $n \in \mathbb{N}$, it holds that
\[
N_2(n) =\{(I_n+S)^{-1} (I_n-S)R \mid  R\in \Sigma_n(\mathbb{Q}),  \; S \in \SS_n(\mathbb{Q}), \; S \mathbf{1}= \mathbf{0}\}.
\]
Here, $\mathbf{1}$ is the all-one vector.
\end{theorem}
However, the group structure of $N_2(n)$ was not considered in \cite{TWZ2025}.
In the main theorem (Theorem~\ref{MT}), we determine its group structure.

Let $K$ be a field, and denote the group of matrices of order $n$ by $M_n(K)$.
Let $Q$ be a positive definite symmetric matrix corresponding to a quadratic form.
We denote the orthogonal group of $Q$ over $K$ of degree $n$ by $O_n(Q,K)$, i.e.,
\[
O_n(Q,K) =\{A \in \GL_{n}(K) \mid A^TQA=Q \}.
\]

\begin{theorem} \label{MT}
Let $n \geq 2$, and  
$Q$ be a diagonal matrix defined by
\begin{align} \label{Q}
Q =\diag \left[1, 3 , \ldots , \frac{n(n-1)}{2}  \right]. 
\end{align}
Then it holds that
\[
N_2(n) \simeq O_{n-1}(Q,\mathbb{Q}).
\]
\end{theorem}

This paper is organized as follows:
In \Cref{sect:proof}, we first show that $N_2(n)$ forms a group under matrix multiplication  (\cref{N2(n)}) and then prove the main theorem (\Cref{MT}).
Finally in \Cref{sect:conclusion}, we give a higher-dimensional generalization of the notion of ellipsoidal design, and conclude this paper with an open question concerning such designs.

%%%%%%%%%%%%%%%%%%%%%%%%%%%%%%%%%%%%%%%%%%%%%%%%%
%%%%%%%%%%%%%%%%%%%%%%%%%%%%%%%%%%%%%%%%%%%%%%%%%
\section{Proof of the main theorem}  \label{sect:proof}
In this section, we prove \cref{MT}.
First, we show that $N_2(n)$ forms a group.

\begin{proposition} \label{N2(n)}
Let $n \geq 2$.
Then it holds that
\[
N_2(n)= \{ A \in O_n(\mathbb{Q}) \mid  A \mathbf{1}=\mathbf{1}\}.
\]
In particular, $N_2(n)$ forms a group under matrix multiplication.
\end{proposition}

\begin{proof}
Let $A$ be in the right hand side.
Since $A \mathbf{1}=\mathbf{1}$ and $AA^T=I_n$ imply the  conditions of $PTE_n$ of degree $1$ and $2$ respectively, 
it follows that $A=I_n^{-1}A \in N_2(n)$.

Suppose $C \in N_2(n)$. Then there exist $A, B \in M_2(\mathbb{Q})$ such that $C=A^{-1}B$ and $[A]=_2^n [B]$.
Since $[A]=_2^n [B]$, we have $AA^T=BB^T$ and $A \mathbf{1}=B\mathbf{1}$.
Therefore, we have $(A^{-1}B)(A^{-1}B)^T=I_n$ and $A^{-1}B\mathbf{1}=\mathbf{1}$.
By the representation of the right hand side, $N_2(n)$ forms a group under matrix multiplication.
\end{proof}

Now we prove our main theorem.

\begin{proof}[{Proof of \cref{MT}}]
By \cref{N2(n)}, we have
\[
N_2(n)= \{ A \in O_n(\mathbb{Q}) \mid  A \mathbf{1}=\mathbf{1}\}.
\]
Let $A \in 
N_2(n)$ . Then for all $\mathbf{x} \in \langle \mathbf{1} \rangle^{\perp}$, we have
\begin{align*}
(A \mathbf{x} , \mathbf{1})=(A \mathbf{x} , A\mathbf{1})=(\mathbf{x} , \mathbf{1})=0
\end{align*}
since $A \mathbf{1}=\mathbf{1}$ and $AA^T=I_n$.
Therefore, $A$ defines an orthogonal transform $\phi:\langle \mathbf{1} \rangle^{\perp} \to  \langle \mathbf{1} \rangle^{\perp}$
that preserves the rationality.
Here, $\langle \mathbf{1} \rangle$ is the $\mathbb{Q}$-vector space 
spanned by $\mathbf{1}$ and $\langle \mathbf{1} \rangle^{\perp} $ is its
orthogonal complement. 

Let $\mathbf{e}_k =[1, \ldots,1, -k, 0, \ldots, 0]^T  \in M_{n,1}(\mathbb{Q})$, whose $(k+1)$-st component is $-k$.
Then $\{\mathbf{e}_k \}_{k=1}^{n-1}$ forms a basis of $\langle \mathbf{1} \rangle^{\perp} $.
Let $B$ be the representation matrix of $\phi$ with respect to the basis $\{\mathbf{e}_k \}_{k=1}^{n-1}$
and 
\[
2Q:=[(\mathbf{e}_i, \mathbf{e}_j) _{ij}] = 2 \diag \left[1, 3, \ldots, \frac{n(n-1)}{2} \right].
\] 
If we denote $\mathbf{x}=\xi_1 \mathbf{e}_1+\dots + \xi_{n-1}\mathbf{e}_{n-1} \in \langle \mathbf{1} \rangle^{\perp} $ and $\boldsymbol{\xi}:=(\xi_1, \ldots, \xi_{n-1})$, then we have
\[
\boldsymbol{\xi}^T 2Q \boldsymbol{\xi}=(\mathbf{x} ,\mathbf{x})=(A\mathbf{x}, A\mathbf{x})=(B \boldsymbol{\xi})^T 2Q (B \boldsymbol{\xi}).
\] 
Thus, we obtain $B^TQB=Q$. Therefore, we have an isomorphism 
\[
N_2(n) \simeq O_{n-1}(Q,\mathbb{Q}), \; A \mapsto B.
\] 
\end{proof}

\begin{remark} 
\cref{MT} can be proven without explicit computation, except for a specific orthonormal basis $\{\mathbf{e}_k \}_{k=1}^{n-1}$.
Indeed, $\phi$ is equivalent to an isometry on $\langle \mathbf{1} \rangle^{\perp}$, i.e., it preserves the quadratic form $2Q$,
and the group of isometries are isomorphic to $O_{n-1}(2Q)=O_{n-1}(Q)$. 
 See \cite[Definition 2.8 and p.\ 26]{Gerstein2008} for details. 
\end{remark}

 In the rest of this section, we examine whether $O_{n-1}(Q,\mathbb{Q})$ in \cref{MT} is conjugate to the usual orthogonal group $O(n-1,\mathbb{Q})$ (\cref{rem:sq}).
To do this, we define the similarity of quadratic forms.

\begin{definition} [{Similarity of quadratic forms}] \label{similar}
Two positive definite symmetric matrices $Q_1$ and $Q_2 \in M_n(\mathbb{Q})$ are {\it equivalent} if there exists 
$g \in \GL_n(\mathbb{Q})$ such that $g^TQ_1g=Q_2$.
Moreover, $Q_1$ and $Q_2$ are {\it similar} if there exists $a \in \mathbb{Q}^{\times}$ such that $Q_1$ and $aQ_2$ are equivalent.
\end{definition}

\begin{remark}
Suppose that $Q_1$ and $Q_2$ are equivalent, say $g^TQ_1g=Q_2$ for some $g \in \GL_n(\mathbb{Q})$. 
Then $O(Q_1,\mathbb{Q})$ and $O(Q_2,\mathbb{Q})$ are isomorphic via $A \mapsto g^{-1}Ag$.
If $Q_1$ and $Q_2$ are similar, then $O(Q_1,\mathbb{Q})$ and $O(Q_2,\mathbb{Q})$ are isomorphic
since $O(aQ_2,\mathbb{Q})=O(Q_2,\mathbb{Q})$ for $a \in \mathbb{Q}^{\times}$.
\end{remark}

\begin{remark} \label{rem:sq}
\begin{enumerate}[labelindent=0pt,itemindent=*]
\item[(i)]
As seen from the proof of \cite[Theorem 1]{Schoenberg1937}, two quadratic forms
$\sum_{i=1}^{n-1} i (i+1) x_i^2$ and 
$\sum_{i=1}^{n-1} y_i^2$
are similar if and only if one of the following conditions holds:
\begin{enumerate}
\item[(a)] $n$ is odd and square.
\item[(b)] $n \equiv 0 \pmod{4}$.
\item[(c)] $n \equiv 2 \pmod{4}$ and $n$ is a sum of two squares.
\end{enumerate}
In these cases, $O_{n-1}(Q,\mathbb{Q})$ is conjugate to $O(n-1,\mathbb{Q})$.
The sequence of $n-1$ satisfying one of the above conditions is given in \cite{OEIS-A096315}.
\end{enumerate}\vspace*{-2\partopsep}
\begin{enumerate}
\item[(ii)]  If none of the above conditions hold, then $Q'$ is not similar to $I_{n-1}$ by \cite[Theorem 1]{Schoenberg1937}.
In this case, $O_{n-1}(Q,\mathbb{Q})$ and $O(n-1,\mathbb{Q})$ are not conjugate by \cite[Lemma 1]{Ono1955}.
 \end{enumerate} 
\end{remark}

\begin{comment}
\begin{theorem} [{\cite[Theorem 1]{Schoenberg1937}}] \label{sim}
Two quadratic forms
$\sum_{i=1}^{n-1} i (i+1) x_i^2$ and 
$\sum_{i=1}^{n-1} y_i^2$
are similar if and only if one of the following conditions holds:
\begin{enumerate}
\item[(i)] $n$ is odd and square.
\item[(ii)] $n \equiv 0 \pmod{4}$.
\item[(iii)] $n \equiv 2 \pmod{4}$ and $n$ is a sum of two squares.
\end{enumerate}
\end{theorem}
The sequence of $n-1$ satisfying \cref{sim} is given in \cite{OEIS-A096315}.

\begin{theorem} [{\cite[Lemma 1]{Ono1955}}] \label{conjugate}
Let $Q_1$, $Q_2 \in M_n(\mathbb{Q})$ be positive definite symmetric matrices. Then
$O_n(Q_1,\mathbb{Q})$ and $O_n(Q_2,\mathbb{Q})$ are conjugate if and only if
$Q_1$ and $Q_2$ are similar.
\end{theorem}

\begin{remark} \label{rem:sq}
\hangindent\leftmargini
\textup{(i)}\hskip\labelsep 
 If $n$ is square, then $N_2(n)$ is conjugate to the usual orthogonal group $O(n-1,\mathbb{Q})$ by \cref{sim} (i) and (iii).
\begin{enumerate}
\item[(ii)]  If $n$ is odd and non-square, then $Q'$ is not similar to $I_{n-1}$ by \cref{sim}.
In this case, $O_{n-1}(Q,\mathbb{Q})$ in \cref{MT} and the usual orthogonal group $O(n-1,\mathbb{Q})$ are not conjugate by \cref{conjugate}.
 \end{enumerate} 
\end{remark}
\end{comment}

%%%%%%%%%%%%%%%%%%%%%%%%%%%%%%%%%%%%%%%%%%%%%%%%%%%%%%%%%%%%%
%%%%%%%%%%%%%%%%%%%%%%%%%%%%%%%%%%%%%%%%%%%%%%%%%%%%%%%%%%%%%
\section{Further remarks and future works}  \label{sect:conclusion}

The following notion is a generalization of {\it spherical design} over $\mathbb{S}^{n-1}$.

 \begin{definition}[{$(n-1)$-dimensional ellipsoidal design}] \label{ED}
Let $Q \in M_n(\mathbb{Q})$ be a positive definite symmetric matrix. For $r \in \mathbb{R}_{>0}$, let $E \subset \mathbb{R}^n$ be an ellipse defined by $\mathbf{x}^TQ\mathbf{x}=r$, where $\mathbf{x} :=(x_1,\ldots,x_n)^T$, and let $\tau$ be an $O_n(Q,\mathbb{R})$-invariant measure on $E$.
A finite nonempty subset $X$ of $E$ is an {\it ellipsoidal $t$-design} if
 \[
 \frac{1}{|X|} \sum_{\mathbf{x} \in X} f(\mathbf{x})=\frac{1}{|E|}\int_E f(\mathbf{x}) d\tau(\mathbf{x})
 \text{\; for all polynomials $f$ of degree $\leq t$.}  
  \]
In particular when $E=\mathbb{S}^{n-1}$ and $Q=I_n$, the point set $X$ is called a {\it spherical $t$-design}. 
 \end{definition}

 An ellipsoidal $t$-design on $E$ is called {\it tight} if 
  \[
  |X| = \begin{cases}
   \binom{n+s-1}{s}+\binom{n+s-2}{s-1} & \text{($t = 2s$)}\\
    2\binom{n+s-1}{s} & \text{($t=2s+1$)}.
  \end{cases}
 \]
The tightness was first defined by Delsarte et al.~\cite{DGS1977} for
spherical designs.
Since an ellipsoid and a sphere are homeomorphic, we can similarly define the tightness for ellipsoidal designs in general.
However, by considering rational points, we can find an essential gap between spheres and ellipses.
Namely, there exists a rational tight design in $\mathbb{R}^2$ only if $(x,y) Q (x,y)^T = x^2 + y^2$ or $x^2 + xy + y^2$ (see \cite[Theorem 4.4]{MS2025+}). 
It is still an open question whether there exists a rational (not necessarily tight) $4$-design on $\mathbb{S}^1$, 
 whereas there exist infinitely many rational tight 
$5$-designs on the ellipse $x^2 + xy + y^2 = 1$. 
Matsumura and Sawa~\cite[Theorem 5.5]{MS2025+}
established a construction method of rational ideal solutions of $\PTE_2$ by utilizing rational tight designs over the ellipse $x^2+xy+y^2=1$.
Since $x^2+xy+y^2=(x+y/2)^2+3(y/2)^2$, the two quadratic forms $x^2+xy+y^2$ and $x^2+3y^2$ are equivalent.
The latter quadratic form appears in the case $n=3$ of \cref{MT}.

We close this paper with the following interesting open question:

\begin{problem} \label{TRED}
Can we construct higher-dimensional tight rational ellipsoidal designs for $Q$ given in \cref{MT}?
For example, can we construct $4$-dimensional tight rational ellipsoidal designs for
$ \diag[1,3,6,10]$?  
\end{problem}
Note that if $n=4$, then $Q$ is similar to $I_3$ by \cref{rem:sq}.
Thus, the first case to be handled is $n=5$.

\noindent {\bf Acknowledgements.}
The authors thank Hiroki Shimakura for helpful comments on the earlier draft.

\end{document}